\documentclass[reqno,10pt]{amsart}

\usepackage{amsmath}
\usepackage{amssymb}
\usepackage{amsthm}
\usepackage{mathrsfs}
\usepackage{geometry}
\usepackage{graphicx}
\usepackage{color}
\usepackage{graphics}
\usepackage[latin1]{inputenc}
\usepackage{verbatim}
\usepackage{tikz}
\usepackage{multirow}
\usepackage{expdlist}

\newcommand{\id}{\mathrm{id}}

\newcommand{\floor}[1]{\lfloor {#1} \rfloor}

\newcommand{\wn}{\mathcal{W}_n}
\newcommand{\br}{(k_1,\ldots,k_n)}

\newcommand{\rk}{\mathrm{rk}}
\newcommand{\block}[4]{\left(\begin{tabular}{c|c} #1 & #2 \\ \hline
#3 & #4 \end{tabular} \right)}

\newcommand{\sts}{\large} 
\newcommand{\SC}{\operatorname{sc}}
\newcommand{\dd}{\operatorname{dd}}

\newcommand{\trap}[1]{%
\begin{tikzpicture}[scale=#1]%
\draw (0,0) -- (3,0);%
\draw (0,0) -- (1,1);%
\draw (1,1) -- (2,1);%
\draw (2,1) -- (3,0);%
\end{tikzpicture}%
}
\newcommand{\strap}{\trap{0.2}}

\numberwithin{equation}{section}

\newtheorem{theorem}{Theorem}
\newtheorem{lemma}{Lemma}

\newtheorem{corollary}{Corollary}

\newtheorem{conjecture}{Conjecture}

\begin{document}

\author{Ilse Fischer \and Lukas Riegler}
\title[Combinatorial Reciprocity for Monotone Triangles]{Combinatorial
Reciprocity for Monotone Triangles} \thanks{Supported by the Austrian Science Foundation FWF, START
grant Y463 and NFN grant S9607-N13.}
\thanks{Fakultät für Mathematik, Universität Wien, Nordbergstraße 15, A-1090
Wien, Austria}
\thanks{E-Mail: ilse.fischer@univie.ac.at, lukas.riegler@univie.ac.at}

\begin{abstract}
The number of Monotone Triangles with bottom row $k_1 < k_2 < \cdots < k_n$ is
given by a polynomial $\alpha(n; k_1,\ldots,k_n)$ in $n$ variables. The
evaluation of this polynomial at weakly decreasing sequences $k_1 \geq k_2
\geq \cdots \geq k_n$ turns out to be interpretable as signed enumeration of new
combinatorial objects called Decreasing Monotone Triangles. There exist
surprising connections between the two classes of objects -- in particular it is
shown that $\alpha(n;1,2,\ldots,n) = \alpha(2n; n,n,n-1,n-1,\ldots,1,1)$. In
perfect analogy to the correspondence between Monotone Triangles and Alternating
Sign Matrices, the set of Decreasing Monotone Triangles with bottom row
$(n,n,n-1,n-1,\ldots,1,1)$ is in one-to-one correspondence with a certain set of
ASM-like matrices, which also play an important role in proving the claimed identity
algebraically. Finding a bijective proof remains an open problem.
\end{abstract}

\maketitle

\section{\sts Introduction}
A \emph{Monotone Triangle} of size $n$ is a triangular array of integers
$(a_{i,j})_{1 \leq j \leq i \leq n}$ 
\[
\begin{array}{ccccccc}
&&& a_{1,1} \\
&& a_{2,1} && a_{2,2} \\
& \rotatebox{75}{$\ddots$} &&&& \ddots \\
a_{n,1} &&\cdots&  &\cdots&& a_{n,n}
\end{array}
\]
with strict increase along rows and weak increase along North-East- and
South-East-diagonals. An example of a Monotone Triangle of size $5$ is
given in Figure \ref{mt-example}.
\begin{figure}[ht]
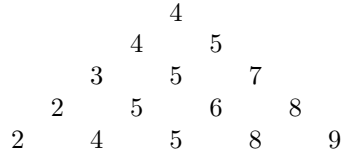

\begin{center}
$
\begin{array}{ccccccccc}
&&&& 4 \\
&&& 4 && 5  \\
&& 3 && 5 && 7 \\
& 2 && 5 && 6 && 8 \\
2 && 4 && 5 && 8 && 9 \\
\end{array}
$
\end{center}
\caption{One of the $16939$ Monotone Triangles with bottom row
$(2,4,5,8,9)$. \label{mt-example}}
\end{figure}
In \cite{FischerNumberOfMT} it has been shown that for each $n \geq 1$, there exists
a polynomial $\alpha(n;k_1,k_2,\ldots,k_n)$ of degree $n-1$ in each of the $n$
variables, such that the evaluation of this polynomial at strictly increasing
sequences $k_1 < k_2 < \cdots < k_n$ yields the number of Monotone Triangles
with prescribed bottom row $\br$ -- for example $\alpha(5;2,4,5,8,9) = 16939$. We now
give an interpretation to the evaluations at weakly decreasing sequences $k_1
\geq k_2 \geq \cdots \geq k_n$:

A \emph{Decreasing Monotone Triangle} (DMT) of size $n$ is a triangular array of
integers $(a_{i,j})_{1 \leq j \leq i \leq n}$ having the
following properties: 
\begin{enumerate}
\item The entries along North-East- and South-East-diagonals are weakly
decreasing.
\item Each integer appears at most twice in a row.
\item Two consecutive rows do not contain the same integer exactly once.
\end{enumerate}

\noindent Note that the weak decrease along diagonals implies a weak decrease
along rows. Let $\wn(k_1,\ldots,k_n)$ denote the set of DMTs with $n$ rows and prescribed bottom row
$(k_1,\ldots,k_n)$. As an example, if the bottom row is $(6,3,3,2,1)$, the
right-most entry of the penultimate row has to be $2$, thus its left neighbour has to
be $2$ too. The second entry has to be $3$ and the first entry may be $5,4$ or $3$. In
total $\mathcal{W}_5(6,3,3,2,1)$ consists of $5$ DMTs (see Figure \ref{dmts}).
\begin{figure}
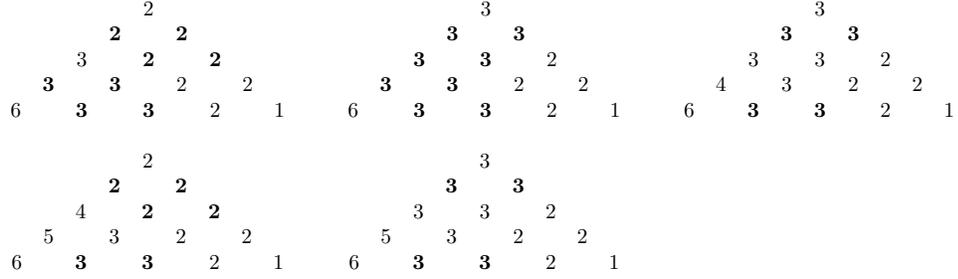

\begin{center}
\scalebox{0.8}{
$
\begin{array}{ccc}
\begin{array}{ccccccccc}
&&&& 2 \\
&&& \mathbf{2} && \mathbf{2}  \\
&& 3 && \mathbf{2} && \mathbf{2} \\
& \mathbf{3} && \mathbf{3} && 2 && 2 \\
6 && \mathbf{3} && \mathbf{3} && 2 && 1 \\
\end{array}
&\quad
\begin{array}{ccccccccc}
&&&& 3 \\
&&& \mathbf{3} && \mathbf{3}  \\
&& \mathbf{3} && \mathbf{3} && 2 \\
& \mathbf{3} && \mathbf{3} && 2 && 2 \\
6 && \mathbf{3} && \mathbf{3} && 2 && 1 \\
\end{array}
&\quad
\begin{array}{ccccccccc}
&&&& 3 \\
&&& \mathbf{3} && \mathbf{3}  \\
&& 3 && 3 && 2 \\
& 4 && 3 && 2 && 2 \\
6 && \mathbf{3} && \mathbf{3} && 2 && 1 \\
\end{array}
\\
\\
\begin{array}{ccccccccc}
&&&& 2 \\
&&& \mathbf{2} && \mathbf{2}  \\
&& 4 && \mathbf{2} && \mathbf{2} \\
& 5 && 3 && 2 && 2 \\
6 && \mathbf{3} && \mathbf{3} && 2 && 1 \\
\end{array}
&\quad
\begin{array}{ccccccccc}
&&&& 3 \\
&&& \mathbf{3} && \mathbf{3}  \\
&& 3 && 3 && 2 \\
& 5 && 3 && 2 && 2 \\
6 && \mathbf{3} && \mathbf{3} && 2 && 1 \\
\end{array}
\end{array}
$
}
\end{center}
\caption{The five Decreasing Monotone Triangles with bottom row
$(6,3,3,2,1)$.\label{dmts}}
\end{figure}
A pair of adjacent identical entries in a row is briefly called
\emph{pair}. A pair is called \emph{duplicate-descendant} (DD), if it is in the
last row, or if the row below contains the same pair. The duplicate-descendants of the DMTs in Figure
\ref{dmts} are marked in boldface.

In Section \ref{sectionDMT} we prove Theorem \ref{mainTheorem}, which states
that the evaluation of $\alpha(n;k_1,\ldots,k_n)$ at weakly decreasing
sequences $k_1 \geq k_2 \geq \cdots \geq k_n$ is a signed enumeration of DMTs
with bottom row $(k_1,k_2,\ldots,k_n)$.
\begin{theorem}
\label{mainTheorem}
Let $n \geq 1$ and $k_1 \geq k_2 \geq \cdots \geq k_n$ be a sequence of weakly
decreasing integers. Then
\begin{equation}
\label{theoremEq}
\alpha(n; k_1,\ldots,k_n) = (-1)^{\binom{n}{2}} \sum_{A \in \wn(k_1,\ldots,k_n)}
(-1)^{\dd(A)},
\end{equation}
where $\dd(A)$ denotes the
total number of duplicate-descendants of $A$.
\end{theorem}
In our example there are four DMTs with an even number of
duplicate-descendants and one with an odd number. This is consistent with
Theorem \ref{mainTheorem} since \mbox{$\alpha(5;6,3,3,2,1) = 3$}. \\

Of special interest are Monotone Triangles with bottom row
$(1,2,\ldots,n)$. This stems from the fact that they are in one-to-one
correspondence with \emph{Alternating Sign Matrices} (ASMs) of size $n$, i.e. $n
\times n$-matrices with entries in $\{0,1,-1\}$, where in each row and column the
non-zero entries alternate in sign and sum up to $1$. As we will see in in
Section \ref{section2ASM}, DMTs with bottom row $(n,n,n-1,n-1,\ldots,1,1)$ are
in correspondence with the following set of matrices: A \emph{$2$-ASM} of size $n$ is a $(2n)\times n$-matrix with
entries in $\{0,1,-1\}$, where in each row the non-zero entries alternate in
sign and sum up to $1$, and in each column the non-zero entries occur in pairs,
such that each partial column sum is in $\{0,1,2\}$ and each column sums up to
$2$. An example of corresponding objects can be seen in Figure \ref{2asm-example}.
\begin{figure}
\begin{center}
\scalebox{0.9}{
$
\begin{array}{ccc}
\left(
\begin{array}{ccccc}
0 & 0 & 1 & 0 & 0 \\
0 & 0 & 1 & 0 & 0 \\
0 & 0 & 0 & 1 & 0 \\
0 & 1 & -1 & 1 & 0 \\
1 & -1 & 1 & 0 & 0 \\
1 & 0 & -1 & 0 & 1 \\
0 & 1 & -1 & 0 & 1 \\
0 & 1 & 0 & 0 & 0 \\
0 & 0 & 1 & 0 & 0 \\
0 & 0 & 1 & 0 & 0 \\
\end{array}
\right)
&
\Leftrightarrow
&
\begin{array}{cccccccccccccccccccc}
&&&&&&&&& 3 \\
&&&&&&&& 3 && 3 \\
&&&&&&& 4 && 3 && 3 \\
&&&&&& 4 && 4 && 3 && 2 \\
&&&&& 4 && 4 && 3 && 3 && 1 \\
&&&& 5 && 4 && 4 && 3 && 1 && 1 \\
&&& 5 && 5 && 4 && 4 && 2 && 1 && 1 \\
&& 5 && 5 && 4 && 4 && 2 && 2 && 1 && 1 \\
& 5 && 5 && 4 && 4 && 3 && 2 && 2 && 1 && 1 \\
5 && 5 && 4 && 4 && 3 && 3 && 2 && 2 && 1 && 1
\end{array}
\end{array}
$
}
\end{center}
\caption{A $2$-ASM of size $5$ and the corresponding DMT.\label{2asm-example}}
\end{figure}
The reason for the interest in DMTs with bottom row $(n,n,n-1,n-1,\ldots,1,1)$
originates from computational experiments indicating that -- surprisingly --
\begin{equation}
\label{alphaConnEq}
\alpha(2n;n,n,n-1,n-1,\ldots,1,1) = \alpha(n;1,2,\ldots,n)
\end{equation}
seems to hold. In Section \ref{sectionConnection} we prove the following
stronger result:
\begin{theorem}
\label{mainTheorem2}
Let $A_{n,i}$ denote the refined ASM-numbers, i.e. the number of ASMs with the
first row's unique $1$ in column $i$. Then
\begin{equation}
\label{motivationEq}
\alpha(2n-1; n-1+i, n-1,n-1,\ldots,1,1) = (-1)^{n-1} A_{n,i}
\end{equation}
holds for $i = 1,\ldots,2n-1$, $n\geq 1$.
\end{theorem}
Having a combinatorial interpretation in terms of DMTs for the left-hand side
of (\ref{alphaConnEq}) and one involving
Monotone Triangles for the right-hand side, the equality demands for a
combinatorial explanation, i.e. a bijective proof. In Section
\ref{sectionOutlook} a first approach towards a bijective proof is given.

\section{\sts Decreasing Monotone Triangles}
\label{sectionDMT}

Removing the bottom row of a Monotone Triangle with $n$ rows yields a Monotone
Triangle with $n-1$ rows. If we want to count the number of Monotone Triangles
with bottom row $k_1 < k_2 < \cdots < k_n$, we can thus determine all admissible
penultimate rows $l_1 < l_2 < \cdots < l_{n-1}$ and count the number of Monotone
Triangles with these bottom rows. It follows that
\begin{equation}
\alpha(n;k_1,\ldots,k_n)= \sum_{\substack{(l_1,\ldots,l_{n-1})\in
\mathbb{Z}^{n-1}, \\
k_1 \leq l_1 \leq k_2 \leq l_2 \leq \cdots \leq k_{n-1} \leq l_{n-1} \leq k_n,
\\ l_i < l_{i+1}}} \alpha(n-1; l_1,\ldots,l_{n-1})
\end{equation}
holds for all $k_1 < k_2 < \cdots < k_n$, $k_i \in \mathbb{Z}$.

This motivates the following inductive definition of a summation operator
$\sum\limits_{(l_1,\ldots,l_{n-1})}^{(k_1,\ldots,k_n)}$ for arbitrary
$(k_1,\ldots,k_n) \in \mathbb{Z}^n$:
\begin{align}
\label{sumOpRec}
\sum_{(l_1,\ldots,l_{n-1})}^{\br} A(l_1,\ldots,l_{n-1}) := 
&\sum_{(l_1,\ldots,l_{n-2})}^{(k_1,\ldots,k_{n-1})}
\sum_{l_{n-1} = k_{n-1}+1}^{k_n} A(l_1,\ldots,l_{n-2},l_{n-1}) \\\notag &+
\sum_{(l_1,\ldots,l_{n-2})}^{(k_1,\ldots,k_{n-2},k_{n-1}-1)}
A(l_1,\ldots,l_{n-2},k_{n-1}), \quad n \geq 2,
\end{align}
with $\sum\limits_{()}^{(k_1)}:=\id$ and the extended definition of ordinary
sums
\begin{equation}
\label{sumExtDef}
\sum_{i=a}^b f(i) := \begin{cases} 0, & \quad b=a-1, \\ -\sum\limits_{i =
b+1}^{a-1} f(i), & \quad b+1 \leq a-1. \end{cases}
\end{equation}
In particular, it follows that
\[
\sum_{(l_1)}^{(k_1,k_2)} f(l_1) = \sum_{()}^{(k_1)} \sum_{l_1=k_1+1}^{k_2}
f(l_1) + \sum_{()}^{(k_1-1)} f(k_1) = \sum_{l_1=k_1}^{k_2} f(l_1), \quad \forall
k_1,k_2 \in \mathbb{Z}.
\]
Note that in the case of strictly
increasing sequences $k_1 < k_2 < \cdots < k_n$ we have that
\[
\sum_{(l_1,\ldots,l_{n-1})}^{\br} = \sum_{\substack{(l_1,\ldots,l_{n-1})\in
\mathbb{Z}^{n-1}, \\
k_1 \leq l_1 \leq k_2 \leq l_2 \leq \cdots \leq k_{n-1} \leq l_{n-1} \leq k_n,
\\ l_i < l_{i+1}}}.
\]
An advantage of the summation operator is that we can now give a recursive
description of $\alpha(n;k_1,\ldots,k_n)$ for arbitrary $(k_1,\ldots,k_n) \in
\mathbb{Z}^n$:
\begin{equation}
\label{alphaRec}
\alpha(n;k_1,\ldots,k_n) = \sum_{(l_1,\ldots,l_{n-1})}^{\br} \alpha(n-1;l_1,\ldots,l_{n-1}).
\end{equation}
The correctness of (\ref{alphaRec}) relies on the extended definition of
summation in (\ref{sumExtDef}). To see this, first note that for every
polynomial $p(x)$ there exists a polynomial $q(x)$ such that $q(x+1)-q(x) =
p(x)$. It follows that
\[
\sum_{i=a}^b p(i) = q(b+1)-q(a), \quad a \leq b, \quad a,b \in \mathbb{Z}.
\]
The crucial observation is, that by using the extended definition
of summation this equality also holds for integers $a > b$. Use induction in
(\ref{sumOpRec}) to see that if $A(l_1,\ldots,l_{n-1})$ is a polynomial in each
$l_i$, then there exists a polynomial $B$ in $n$ variables satisfying
\[
B(k_1,\ldots,k_n) = \sum_{(l_1,\ldots,l_{n-1})}^{\br} A(l_1,\ldots,l_{n-1}),
\quad (k_1,\ldots,k_n) \in \mathbb{Z}^n.
\]
Firstly, this implies the existence of a polynomial, which is equal to the
right-hand side of (\ref{alphaRec}) for arbitrary $(k_1,\ldots,k_n) \in \mathbb{Z}^n$.
As (\ref{alphaRec}) holds for $k_1 < k_2 < \cdots < k_n$ and, furthermore, a
polynomial in $n$ variables is uniquely determined by the values on $\{(k_1,\ldots,k_n)\in \mathbb{Z}^n: k_1 < k_2 < \cdots < k_n\}$, it follows that (\ref{alphaRec})
holds for arbitrary $(k_1,\ldots,k_n) \in \mathbb{Z}^n$.
\noindent Secondly it tells us that instead of definition (\ref{sumOpRec}), we
can use any other inductive definition to get the same polynomial, as long as it is based on
the extended definiton of ordinary summation and coincides
with
\[
\sum_{\substack{(l_1,\ldots,l_{n-1})\in
\mathbb{Z}^{n-1}, \\
k_1 \leq l_1 \leq k_2 \leq l_2 \leq \cdots \leq k_{n-1} \leq l_{n-1} \leq k_n,
\\ l_i < l_{i+1}}}
\]
for all strictly increasing integer sequences $k_1 < k_2 < \cdots < k_n$. In
particular, we will also use
\begin{align}
\label{sumOpRec2}
\sum_{(l_1,\ldots,l_{n-1})}^{\br} A(l_1,\ldots,l_{n-1}) =
& \sum_{(l_1,\ldots,l_{n-2})}^{(k_1,\ldots,k_{n-1})} \sum_{l_{n-1} =
k_{n-1}}^{k_n} A(l_1,\ldots,l_{n-2},l_{n-1}) \\\notag &-
\sum_{(l_1,\ldots,l_{n-3})}^{(k_1,\ldots,k_{n-2})}
A(l_1,\ldots,l_{n-3},k_{n-1},k_{n-1}), \quad n \geq 3,
\end{align}
and
\begin{align}
\label{sumOpRec3}
\sum_{(l_1,\ldots,l_{n-1})}^{\br} A(l_1,\ldots,l_{n-1}) = 
& \sum_{(l_2,\ldots,l_{n-1})}^{(k_2,\ldots,k_{n})} \sum_{l_{1} =
k_1}^{k_2-1} A(l_1,l_2,\ldots,l_{n-1}) \\\notag &+
\sum_{(l_2,\ldots,l_{n-1})}^{(k_2+1,k_3,\ldots,k_n)}
A(k_2,l_2,\ldots,l_{n-1}), \quad n \geq 2.
\end{align}
Equipped with the summation operator, let us first consider the case that the
bottom row $\br$ is weakly decreasing and contains three identical
entries. By definition, $\wn(k_1,\ldots,k_n)$ is the empty set, and thus the
right-hand side of (\ref{theoremEq}) is zero. Therefore we have to show that the
polynomial $\alpha(n;k_1,\ldots,k_n)$ also vanishes in this case.
\begin{lemma}
Let $n \geq 3$ and $A(l_1,\ldots,l_{n-1})$ be a polynomial in each variable
satisfying
\[
A(l_1,\ldots,l_{i-1},l_i,l_i,l_i,l_{i+3},\ldots,l_{n-1}) = 0,
\quad i=1,\ldots,n-3.
\]
The polynomial $B(k_1,\ldots,k_n):=
\sum\limits_{(l_1,\ldots,l_{n-1})}^{(k_1,\ldots,k_n)} A(l_1,\ldots,l_{n-1})$
then satisfies
\[
B(k_1,\ldots,k_{i-1},k_i,k_i,k_i,k_{i+3},\ldots,k_n) = 0, \quad i=1,\ldots,n-2.
\]
In particular
\begin{equation}
\alpha(n; k_1,\ldots,k_{i-1},k_i,k_i,k_i,k_{i+3},\ldots,k_n) = 0, \quad
i=1,\ldots,n-2.
\end{equation}
\end{lemma}
\begin{proof}
The case $n=3$ is easy to check using (\ref{sumOpRec}). There are three cases
to check for $n \geq 4$ depending on whether there are no, one or at least two
entries to the right of the three identical entries:
\begin{itemize}
    \item $i=n-2$: 
    Applying (\ref{sumOpRec}) to obtain the first equality, (\ref{sumOpRec2})
    for the second equality and the Lemma's
    assumption for the third, we see that
    \begin{multline*}
    B(k_1,\ldots,k_{n-3},k_{n-2},k_{n-2},k_{n-2}) 
    =\sum_{(l_1,\ldots,l_{n-2})}^{(k_1,\ldots,k_{n-3},k_{n-2},k_{n-2}-1)}
    A(l_1,\ldots,l_{n-2},k_{n-2}) \\
    = - \sum_{(l_1,\ldots,l_{n-4})}^{(k_1,\ldots,k_{n-3})}
    A(l_1,\ldots,l_{n-4},k_{n-2},k_{n-2},k_{n-2}) = 0.
    \end{multline*}
    \item $i=n-3$: Applying recursion (\ref{sumOpRec2}) yields
    \begin{multline*}
    B(k_1,\ldots,k_{n-4},k_{n-3},k_{n-3},k_{n-3},k_{n}) \\
    =\sum_{(l_1,\ldots,l_{n-2})}^{(k_1,\ldots,k_{n-4},k_{n-3},k_{n-3},k_{n-3})}
    \sum_{l_{n-1}=k_{n-3}}^{k_n} A(l_1,\ldots,l_{n-2},l_{n-1}) - \\
    -\sum_{(l_1,\ldots,l_{n-3})}^{(k_1,\ldots,k_{n-4},k_{n-3},k_{n-3})}
    A(l_1,\ldots,l_{n-3},k_{n-3},k_{n-3}). \\ 
    \end{multline*}
    Note that $A'(l_1,\ldots,l_{n-2}):= \sum\limits_{l_{n-1}=k_{n-3}}^{k_n}
    A(l_1,\ldots,l_{n-2},l_{n-1})$ satisfies the Lemma's hypothesis. By
    induction the first sum vanishes, and using (\ref{sumOpRec}) shows that
    the second sum vanishes too.
    \item $1 \leq i \leq n-4$: Applying (\ref{sumOpRec}) and the
    induction hypothesis as in the previous case implies that
    $B(k_1,\ldots,k_{i-1},k_i,k_i,k_i,k_{i+3},\ldots,k_n) = 0$.
   \end{itemize}
In particular, this shows that $\alpha(n;k_1,\ldots,k_n) = 0$, whenever
there are three consecutive identical entries among $(k_1,\ldots,k_n)$.
\end{proof}

We can now restrict ourselves to the case of weakly decreasing sequences
$k_1 \geq k_2 \geq \cdots \geq k_n$, which contain each integer at most
twice. First, note that condition $(3)$ of DMTs is equivalent to the following
condition:
\begin{itemize}
\item[$(3')$] If an entry is equal to its South-East-neighbour and smaller
than its South-West-neighbour, then the entry has a right neighbour
and is equal to it. If an entry is equal to its South-West-neighbour and
greater than its South-East-neighbour, then the entry has a left neighbour and is equal to
it.
\end{itemize}
\noindent In the following, an entry strictly smaller than the
South-West-neighbour and strictly larger than the South-East-neighbour will be called \emph{newcomer},
i.e. an entry in row $i$, $1 \leq i < n$, which is not appearing in row $i+1$.

The proof of Theorem \ref{mainTheorem} consists of the following parts:
In Lemma \ref{sumOpDMT} we show that applying the summation operator
$\sum\limits_{(l_1,\ldots,l_{n-1})}^{(k_1,\ldots,k_n)}$ is a signed summation
over all possible penultimate rows of DMTs with bottom row $(k_1,\ldots,k_n)$.
Applying this inductively we see in Corollary \ref{mainProofStepOne} that
$\alpha(n; k_1,\ldots,k_n)$ is the signed summation over all
DMTs with bottom row $(k_1,\ldots,k_n)$, where the sign is determined by the
total number of pairs and newcomers in the DMT without the bottom row. Finally, in
Lemma \ref{statConnLemma} we show that the parity of this statistic is equal
to the parity of the statistic claimed in Theorem \ref{mainTheorem}.

\begin{lemma}
\label{sumOpDMT}
Let $(k_1,\ldots,k_n)$ be a weakly decreasing sequence of integers with each
element appearing at most twice, and let
$\mathcal{P}(k_1,\ldots,k_n)$ denote the set of $(n-1)$-st rows of elements in
$\wn\br$. Then, for every polynomial $A(l_1,\ldots,l_{n-1})$ we have
\begin{equation}
\label{dmtSumOpEq}
\sum_{(l_1,\ldots,l_{n-1})}^{(k_1,\ldots,k_n)} A(l_1,\ldots,l_{n-1}) =
\sum_{(l_1,\ldots,l_{n-1}) \in \mathcal{P}(k_1,\ldots,k_n)}
(-1)^{\SC(\mathbf{k};\mathbf{l})} A(l_1,\ldots,l_{n-1}), \quad n \geq 2,
\end{equation}
where the sign-change function $\SC(\mathbf{k};\mathbf{l}) :=
\SC(k_1,\ldots,k_n;l_1,\ldots,l_{n-1})$ is the number of pairs in
$(l_1,\ldots,l_{n-1})$ plus the number of newcomers.
\end{lemma}
\begin{proof}
It is instructive to check the Lemma for $n=2,3$. For $n \geq 4$, distinguish
between the case $k_{n-1} > k_n$ and $k_{n-1} = k_n$: \\

\noindent\underline{Case $1$}:
If $k_{n-1} > k_n$, recursion (\ref{sumOpRec2}) of the summation operator,
(\ref{sumExtDef}) and the induction hypothesis yield
\begin{multline*}
\sum_{(l_1,\ldots,l_{n-1})}^{(k_1,\ldots,k_n)} A(l_1,\ldots,l_{n-1}) \\ =
-\sum_{(l_1,\ldots,l_{n-2})}^{(k_1,\ldots,k_{n-1})}
\sum_{l_{n-1}=k_n+1}^{k_{n-1}-1} A(l_1,\ldots,l_{n-1}) -
\sum_{(l_1,\ldots,l_{n-3})}^{(k_1,\ldots,k_{n-2})}
A(l_1,\ldots,l_{n-3},k_{n-1},k_{n-1}) \\
= \sum_{(l_1,\ldots,l_{n-2}) \in \mathcal{P}(k_1,\ldots,k_{n-1})}
(-1)^{\SC(k_1,\ldots,k_{n-1};l_1,\ldots,l_{n-2})+1}
\sum_{l_{n-1}=k_n+1}^{k_{n-1}-1} A(l_1,\ldots,l_{n-1}) \\ +
\sum_{(l_1,\ldots,l_{n-3})\in \mathcal{P}(k_1,\ldots,k_{n-2})}
(-1)^{\SC(k_1,\ldots,k_{n-2};l_1,\ldots,l_{n-3})+1}
A(l_1,\ldots,l_{n-3},k_{n-1},k_{n-1}). 
\end{multline*}
In order to see that this is indeed equal to the right-hand side of
(\ref{dmtSumOpEq}) let us show that for $k_{n-1}> k_n$
\begin{equation}
\label{splitCase1}
\mathcal{P}(k_1,\ldots,k_n) = \bigcup_{l_{n-1} = k_n+1}^{k_{n-1}-1}
\mathcal{P}(k_1,\ldots,k_{n-1}) \times \{l_{n-1}\} \;\cup\;
\mathcal{P}(k_1,\ldots,k_{n-2}) \times \{(k_{n-1},k_{n-1})\}
\end{equation}
holds. By definition of DMTs the value
of $l_{n-1}$ in all rows of $\mathcal{P}(k_1,\ldots,k_n)$ is either strictly
between $k_{n-1}$ and $k_n$ or equal to $k_{n-1}$.\\

\noindent\underline{Case $1.1$}: If $k_{n-1} > l_{n-1} > k_{n}$ and $k_{n-2} >
k_{n-1}$, then $k_{n-2} \geq l_{n-2} > k_{n-1}$ on both sides of
(\ref{splitCase1}).  If $k_{n-2}=k_{n-1}$, then $l_{n-2} = k_{n-2}$ on both
sides. Since $l_{n-1}$ is a newcomer, we have
$\SC(k_1,\ldots,k_n;l_1,\ldots,l_{n-1}) =
\SC(k_1,\ldots,k_{n-1};l_1,\ldots,l_{n-2})+1$. \\

\noindent\underline{Case $1.2$}: If $l_{n-1} = k_{n-1}$, then
by condition $(3')$ of DMTs, we have $l_{n-2} = k_{n-1}$. If $k_{n-3} >
k_{n-2}$, then $l_{n-3} > k_{n-2}$ on both sides: For the right-hand side of
(\ref{splitCase1}) this is clear by
definition. For the left-hand side note that $k_{n-3} > l_{n-3} = k_{n-2}$
would -- by condition $(3')$ -- imply that $l_{n-3} = l_{n-2} = l_{n-1}$,
contradiction. If $k_{n-3} = k_{n-2}$, then $l_{n-3}
= k_{n-2}$ on both sides of (\ref{splitCase1}). Note that, if $k_{n-3} =
k_{n-2}$, then by the Lemma's assumption $k_{n-2} > k_{n-1}$. Thus, condition
$(2)$ of DMTs imposes the same restrictions on both sides. The pair
$(l_{n-2},l_{n-1})$ contributes one sign-change, and thus $\SC(k_1,\ldots,k_n;l_1,\ldots,l_{n-1})=\SC(k_1,\ldots,k_{n-2};l_1,\ldots,l_{n-3})+1$.
\\

\noindent\underline{Case $2$}: If $k_{n-1} = k_n$, recursion (\ref{sumOpRec})
and the induction hypothesis yield
\begin{multline*}
\sum_{(l_1,\ldots,l_{n-1})}^{(k_1,\ldots,k_{n-1},k_{n-1})} A(l_1,\ldots,l_{n-1})
= \sum_{(l_1,\ldots,l_{n-2})}^{(k_1,\ldots,k_{n-2},k_{n-1}-1)} A(l_1,\ldots,l_{n-2},k_{n-1}) \\
=\sum_{(l_1,\ldots,l_{n-2})\in \mathcal{P}(k_1,\ldots,k_{n-2},k_{n-1}-1)}
(-1)^{\SC(k_1,\ldots,k_{n-2},k_{n-1}-1;l_1,\ldots,l_{n-2})}
A(l_1,\ldots,l_{n-2},k_{n-1}) \\ =\sum_{(l_1,\ldots,l_{n-1}) \in
\mathcal{P}(k_1,\ldots,k_{n-2},k_{n-1},k_{n-1})}
(-1)^{\SC(\mathbf{k};\mathbf{l})} A(l_1,\ldots,l_{n-1}).
\end{multline*}
For the last equality let us show that 
\begin{equation}
\label{splitCase2}
\mathcal{P}(k_1,\ldots,k_{n-2},k_{n-1},k_{n-1}) =
\mathcal{P}(k_1,\ldots,k_{n-2},k_{n-1}-1) \times \{k_{n-1}\}.
\end{equation}
The entry $l_{n-1}$ is equal to $k_{n-1}$ on both sides of (\ref{splitCase2})
and does not contribute a sign-change. By the Lemma's assumption
$k_{n-2} > k_{n-1}$ and hence $k_{n-2} \geq l_{n-2} \geq k_{n-1}$ with
condition $(2)$ of DMTs again imposing the same restrictions on both sides. Note
that $\SC(k_1,\ldots,k_{n-2},k_{n-1},k_{n-1};l_1,\ldots,l_{n-1})=
\SC(k_1,\ldots,k_{n-2},k_{n-1}-1;l_1,\ldots,l_{n-2})$ indeed holds: If $k_{n-2}
> l_{n-2} > k_{n-1}$, then $l_{n-2}$ is a newcomer on both sides, and if
$l_{n-2} = k_{n-1}$, then $l_{n-2}$ is in a pair on the left-hand side and a
newcomer on the right-hand side.
\end{proof}

Extend the domain of the sign-changes function $\SC$ to DMTs by defining 
\[
\SC(A):= \sum_{i=1}^{n-1} \SC(a_{i+1,1},\ldots,{a_{i+1,i+1};a_{i,1},\ldots,a_{i,i}}),
\]
where $A=(a_{i,j})_{1\leq j \leq i \leq n}$ is a DMT with $n$ rows. Applying
Lemma \ref{sumOpDMT} to $\alpha(n;k_1,\ldots,k_n)$ establishes the connection
between $\alpha(n;k_1,\ldots,k_n)$ and $\mathcal{W}_n(k_1,\ldots,k_n)$:
\begin{corollary}
\label{mainProofStepOne}
Let $k_1 \geq k_2 \geq \cdots \geq k_n$. Then
\begin{equation}
\label{mainProofStepOneEq}
\alpha(n;k_1,\ldots,k_n) = \sum_{A \in \wn(k_1,\ldots,k_n)} (-1)^{\SC(A)}
\end{equation} 
holds for $n \geq 1$.
\end{corollary}
\begin{proof}
The case $n=1$ is trivial. For $n \geq 2$ apply (\ref{alphaRec})
together with Lemma \ref{sumOpDMT} and the induction hypothesis to see that
\begin{multline*}
\alpha(n;k_1,\ldots,k_n) = \sum_{(l_1,\ldots,l_{n-1}) \in
\mathcal{P}(k_1,\ldots,k_n)} (-1)^{\SC(\mathbf{k};\mathbf{l})} \alpha(n-1;l_1,\ldots,l_{n-1}) \\ = \sum_{(l_1,\ldots,l_{n-1}) \in \mathcal{P}(k_1,\ldots,k_n)} (-1)^{\SC(\mathbf{k};\mathbf{l})} \sum_{A \in \mathcal{W}_{n-1}(l_1,\ldots,l_{n-1})} (-1)^{{sc}(A)} 
= \sum_{A \in \mathcal{W}_n(k_1,\ldots,k_n)} (-1)^{\SC(A)}. \\
\end{multline*}
\end{proof}
In order to complete the proof of Theorem \ref{mainTheorem}, it remains to
be shown that the two statistics $\SC(A)$ and $\binom{n}{2}+\dd(A)$ have the
same parity.
\begin{lemma}
\label{statConnLemma}
Each $A \in \wn\br$ satisfies
\[
(-1)^{\SC(A)} = (-1)^{\binom{n}{2} + \dd(A)}, \quad n \geq 1.
\]
\end{lemma}
\begin{proof}
By definition of DMTs, if a row contains an integer $x$ exactly once, then the
row below contains $x$ either not at all -- i.e. $x$ is a newcomer -- or twice.
In the latter case let us call $x$ a \emph{peak}. Let $p(A)$ denote the number
of peaks in $A$, $n(A)$ denote the number of newcomers in $A$ and $p_i(A)$
denote the number of pairs in the $i$-th row of $A$. Since newcomers and
peaks are by definition not in the bottom row and every entry of a DMT not in
the bottom row is either a peak, a newcomer or in a pair, it follows that
\[
\binom{n}{2} = p(A) + n(A) + 2 \sum_{i=1}^{n-1} p_{i}(A) \equiv p(A) + n(A) \mod
2.
\]
Let us call a pair $(x,x)$ a \emph{base-pair}, if it is located in the bottom
row or the row below contains $x$ exactly once. Let $b(A)$ denote the number of base-pairs in $A$. By
definition, duplicate-descendants are those pairs, which are either in the
bottom row or not a base-pair. Further note that the set of
peaks and the set of base-pairs are in one-to-one correspondence (see Figure
\ref{dmtStat}). Hence, we have
\[
\dd(A) = \sum_{i=1}^n p_i(A) - b(A) + p_n(A) \equiv \sum_{i=1}^{n-1} p_i(A) -
p(A) \mod 2.
\]
Since $\SC(A) = n(A) + \sum_{i=1}^{n-1} p_i(A)$ it follows that
\[
\binom{n}{2} \equiv \sum_{i=1}^{n-1} p_i(A) - \dd(A) + n(A) = \SC(A) - \dd(A)
\mod 2.
\]
\end{proof}

\begin{figure}
\begin{center}
$
\begin{array}{cccccccccccccccc}
&&&&&&& 6 \\
&&&&&& 6 && 6 \\
&&&&& 6 && 6 && 3 \\
&&&& 6 && 6 && 3 && 3 \\
&&& 6 && 6 && 4 && 3 && 3 \\
&& 7 && 6 && 5 && 3 && 3 && 2 \\
& 7 && 7 && 5 && 5 && 3 && 2 && 2 \\
8 && 7 && 5 && 5 && 3 && 3 && 2 && 1
\end{array}
$
\includegraphics[width=8cm]{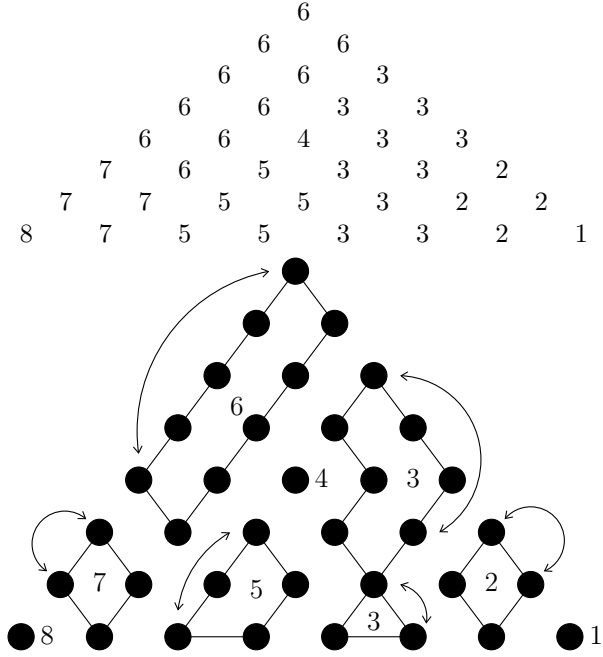}
\caption{A DMT and its structural decomposition \label{dmtStat}}
\end{center}
\end{figure}

\section{\sts DMTs and 2-ASMs}
\label{section2ASM}

Recall the bijection transforming an ASM $(a_{i,j})_{1\leq
i,j \leq n}$ into a Monotone Triangle with bottom row $(1,2\ldots,n)$: The
$i$-th row of the Monotone Triangle contains an entry $j$, iff $b_{i,j} :=
\sum_{k=1}^i a_{k,j} = 1$. In our case let $b_{i,j}$ be the number of entries
$j$ in row $i$ of a DMT with bottom row $(n,n,n-1,n-1,\ldots,1,1)$. Define a
$(2n)\times n$-matrix $A$, such that $\sum_{k=1}^i a_{k,j} = b_{i,j}$, i.e.
$a_{1,j} := b_{1,j}$ and $a_{i,j} := b_{i,j}-b_{i-1,j}$, $2 \leq i \leq 2n$.
Figure \ref{2asm-example} contains an example of a $2$-ASM of size $5$ and its
corresponding DMT.

In order to describe the impact of the change from Monotone Triangles to
DMTs on the level of matrices, one may consider the following equivalent
definition of ASMs: An Alternating Sign Matrix (ASM) of size $n$ is an $n
\times n$-matrix, where each row and each column is a word of the ASM-machine
in Figure \ref{asmCol}. The semantics of the machine is the following: 
When generating a row/column of an Alternating Sign Matrix, the initial state is
$\Sigma = 0$, i.e. the partial row/column sum is $0$. One may then stay at the
state taking the $0$-loop or transit to the state $\Sigma = 1$ by taking the
edge labelled with $1$. In the state $\Sigma = 1$ -- i.e. the partial row/column sum
is currently equal to $1$ -- one may either stay at the state by taking the
$0$-loop or transit back to the state $\Sigma = 0$ taking the edge labelled with
$-1$. As the row/column sum is equal to $1$, in the end one has to be at the
state $\Sigma = 1$.
\begin{figure}
\begin{center}
\includegraphics[width=8cm]{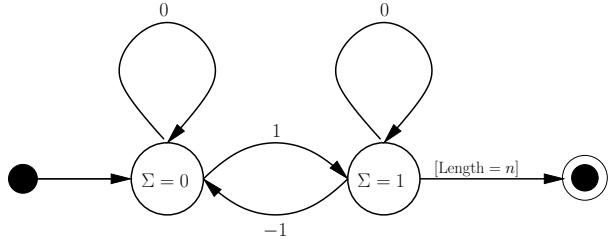}
\caption{ASM-machine generating the rows and columns of ordinary
ASMs of size $n$.\label{asmCol}}
\end{center}
\end{figure}
From this vantage point we can give the following equivalent definition of
$2$-ASMs: A $2$-ASM of size $n$ is a $(2n) \times n$-matrix, where each row is a
word of the ASM-machine in Figure \ref{asmCol} and each column is a word of the
$2$-ASM-machine in Figure \ref{dmtCol}.
\begin{figure}
\begin{center}
\includegraphics[width=8cm]{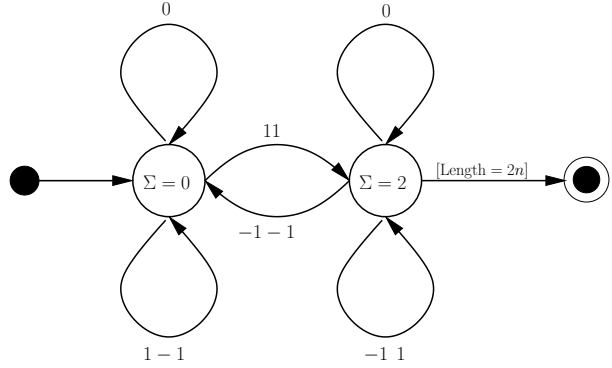}
\caption{$2$-ASM machine generating the columns of $2$-ASMs of
size $n$.\label{dmtCol}}
\end{center}
\end{figure}
\begin{theorem}
\label{dmtBijection}
The set $\mathcal{W}_{2n}(n,n,n-1,n-1,\ldots,1,1)$ is in one-to-one
correspondence with the set of $2$-ASMs of size $n$.
\end{theorem}
\begin{proof}
Let $A = (a_{i,j})_{\substack{i=1,\ldots,2n\\ j=1,\ldots,n}}$ be a $2$-ASM of
size $n$. Define a matrix $B = (b_{i,j})_{\substack{i=1,\ldots,2n\\
j=1,\ldots,n}}$ by $b_{i,j} = \sum_{k=1}^{i} a_{k,j}$. Each of the first $i$
rows of $A$ sum up to $1$, thus the $i$-th row of $B$ sums up to $i$. By
definition of the $2$-ASM-machine, the first $i$ entries of each column sum up
to $0,1$ or $2$, thus $b_{i,j} \in \{0,1,2\}$. Hence, one can define a
triangular array $T$ of integers in $\{1,2,\ldots,n\}$ with $2n$ rows, where
each row is weakly decreasing, by specifying that row $i$ contains $b_{i,j}$
entries $j$. As $b_{i,j} \in \{0,1,2\}$ each row
of $T$ contains an integer at most twice. The fact that the column sums are
$2$ implies that the bottom row of $T$ is $(n,n,n-1,n-1,\ldots,1,1)$.

According to the terminology of duplicate-descendants, let us call the last
condition of DMTs, the No-Single-Descendant-Condition (NSD-Condition). Suppose
$T$ violates the NSD-condition of Decreasing Monotone Triangles in rows $i$ and $i+1$ for the first time. This would mean that there exists a
column $j$ in $B$ such that $b_{i,j} = b_{i+1,j} = 1$. If $i=1$, the $j$-th
column of $A$ has to start with $1,0$, which is not allowed by the
$2$-ASM-machine. If $i>1$ , then the minimality of $i$ implies that $b_{i-1,j}
\in \{0,2\}$. But $b_{i-1,j} = 0$ would mean that after $i-1$ entries in the
column $j$ of $A$, one is at the state $\Sigma = 0$ of the $2$-ASM machine. As
$b_{i,j} = b_{i+1,j} = 1$, the next two column entries of $A$ are $1,0$
contradicting the definition of the $2$-ASM machine. Similarly, $b_{i-1,j} = 2$ together with $b_{i,j} = b_{i+1,j} = 1$ implies being at the
state $\Sigma = 2$ and generating $-1,0$ as next two entries, contradiction.

The weak decrease along diagonals of $T$ can be deduced from the alternating
sign property of each row analagously to the bijection proof between ordinary
ASMs and Monotone Triangles. Hence, $T
\in \mathcal{W}_{2n}(n,n,n-1,n-1,\ldots,1,1)$.

In order to see that the mapping is indeed a bijection we have to show that the
obvious candidate for the inverse mapping yields a $2$-ASM for each
$T\in\mathcal{W}_{2n}(n,n,n-1,n-1,\ldots,1,1)$: So, given $T \in
\mathcal{W}_{2n}(n,n,n-1,n-1,\ldots,1,1)$, define the $(2n)\times n$ matrices
$B$ and $A$ with $b_{i,j}$ being the number of entries $j$ in row $i$ of $T$,
$a_{1,j} = b_{1,j}$ and $a_{i,j} = b_{i,j}-b_{i-1,j}$ for $2 \leq i \leq 2n$.

By definition of $B$ and DMTs we know that $b_{i,j} \in \{0,1,2\}$ and thus
$a_{i,j} \in \{0,1,-1,2,-2\}$. Due to the monotony in DMTs two consecutive
equal entries in a row imply that their interlaced neighbour in the row above
and below has the same value, and thus $a_{i,j} \in \{0,1,-1\}$. The property,
that each row of $A$ sums up to $1$ with the non-zero entries alternating in
sign can again be obtained in a similar manner as for ASMs and ordinary monotone
triangles.

It remains to be shown that each column is a word of the $2$-ASM-machine:
Consider column $j$ of the matrix $A$ or equivalently the entries $j$ in $T$
from top to bottom. If the $i$-th partial column sum of column $j$ in $A$ is
$0$, there are by construction no entries $j$ in row $i$ of $T$. Regarding the
number of entries $j$ in row $i+1$, there are two possibilities:
Either there is still no entry $j$ in row $i+1$, i.e. take the $0$ loop-edge at
state $\Sigma = 0$ of the $2$-ASM-machine, or there exists exactly one entry
$j$. In the latter case we know from the NSD-condition that in the subsequent
row, there is no entry $j$ or two entries $j$, i.e. take the
$(1,-1)$ loop or transit to the state $\Sigma = 2$ taking the $(1,1)$ edge.

If the $i$-th partial column sum of the $j$-th column in $A$ is
$2$, then there are exactly two entries $j$ in row $i$ of $T$. Analogously, the
row below contains either two entries $j$ again, i.e. take the $0$-loop at state
$\Sigma=2$, or the row below contains exactly one entry $j$. In the latter case, the NSD-condition implies that the
subsequent row contains either $2$ or $0$ entries $j$, i.e. take $(-1,1)$-loop
or $(-1,-1)$-edge. As the bottom row of $T$ is $(n,n,n-1,n-1,\ldots,1,1)$ all
columns of $A$ sum up to $2$, and thus every column is a word of the $2$-ASM-machine.
\end{proof}

\section{\sts Connections between Monotone Triangles and DMTs}
\label{sectionConnection}

Starting point for the content of this section was the empirical observation
of (\ref{alphaConnEq}). The algebraic proof in this
section uses a methodology, which has been successfully applied in
\cite{FischerNewRefProof}, where the Refined ASM Theorem is reproven: The number
$A_{n,i}$ of ASMs of size $n$ with the first column's unique $1$ situated in row
$i$ is equal to
\[
\binom{n+i-2}{n-1}\frac{(2n-i-1)!}{(n-i)!} \prod_{k=0}^{n-2}
\frac{(3k+1)!}{(n+k)!}.
\]
The methodology consists of the following steps:
\begin{itemize}
  \item Denoting $E_x f(x) := f(x+1)$ and $\Delta_x f(x) := f(x+1) - f(x) = (E_x
  - \id ) f(x)$, show that \[ A_{n,i} = (-1)^{i-1} \Delta_{k_1}^{i-1} \alpha(n;
  k_1,2,3,\ldots,n)|_{k_1=2}. \]
  \item Derive a system of linear equations for $A_{n,i}$ implying that
  $(A_{n,i})_{i=1,\ldots,n}$ is an eigenvector of a certain matrix.
  \item Show that the corresponding eigenspace is one-dimensional.
  \item Show that the conjectured numbers are in this eigenspace too.
  \item Determine the constant factor by which they differ to be $1$.
\end{itemize}
In this section we prove the following theorem using this
methodology.
\begin{theorem}
\label{wniTheorem}
The numbers
\[
W_{n,i}:=\Delta_{k_1}^{i-1} \alpha(2n-1;
k_1,n-1,n-1,n-2,n-2,\ldots,1,1)|_{k_1=n}
\]
are given by
\begin{equation}
\label{wniTheoremEq}
W_{n,i}=\sum_{l=1}^{i} \binom{i-1}{l-1}(-1)^{n+i+l-1}A_{n,l}
\end{equation}
for all $i=1,\ldots,2n-1$, $n\geq 1$.
\end{theorem}
As motivation on how (\ref{wniTheoremEq}) can be
guessed, let us first note that Theorem \ref{wniTheorem} and
Theorem \ref{mainTheorem2} are equivalent: On the one hand, (\ref{motivationEq})
implies for $i=1,\ldots,2n-1$ that 
\begin{multline*}
W_{n,i} = (E_{k_1}-\id)^{i-1} \alpha(2n-1;k_1,n-1,n-1,\ldots,1,1)|_{k_1=n} \\
= \sum_{l=0}^{i-1} \binom{i-1}{l} (-1)^{i-1-l}
\alpha(2n-1;n+l,n-1,n-1,\ldots,1,1) 
 = \sum_{l=1}^{i} \binom{i-1}{l-1} (-1)^{n+i+l-1} A_{n,l}.
\end{multline*}
Conversely, for $k=1,\ldots,2n-1$ we have that
\begin{multline*}
\alpha(2n-1; n-1+k, n-1,n-1,\ldots,1,1) \\
= E_{k_1}^{k-1} \alpha(2n-1; k_1,n-1,n-1,\ldots,1,1)|_{k_1=n} \\
= (\Delta_{k_1}+\id)^{k-1} \alpha(2n-1; k_1,n-1,n-1,\ldots,1,1)|_{k_1=n} \\
= \sum_{i=0}^{k-1} \binom{k-1}{i} W_{n,i+1} = \sum_{i=0}^{k-1} \binom{k-1}{i}
\sum_{l=1}^{i+1} \binom{i}{l-1}(-1)^{n+i+l} A_{n,l} \\
= \sum_{l=1}^{k} (-1)^{n+l} A_{n,l} \sum_{i=l-1}^{k-1} \binom{k-1}{i}
 \binom{i}{l-1}(-1)^{i}.
\end{multline*}
The inner sum can be simplified using $\binom{n}{k} =
\binom{n}{n-k}$, $\binom{a}{b}=(-1)^b \binom{b-a-1}{b}$ and the Chu-Vandermonde
convolution $\sum_k \binom{r}{m+k}\binom{s}{n-k}=\binom{r+s}{m+n}$
(see \cite{ConcreteMath}, $p.169$):
\begin{multline*}
\sum_{i=l-1}^{k-1} \binom{k-1}{i} \binom{i}{l-1}(-1)^{i} = \sum_{i=l-1}^{k-1}
\binom{k-1}{k-i-1} \binom{i}{i-l+1}(-1)^{i} \\
= (-1)^{l-1} \sum_{i=l-1}^{k-1} \binom{k-1}{k-i-1} \binom{-l}{i-l+1} =
(-1)^{l-1} \binom{k-l-1}{k-l}.
\end{multline*}
Since $\binom{a-1}{a}=0$ for all $a \in \mathbb{Z} \backslash \{0\}$, the only
summand not vanishing is $l=k$, yielding the claimed equation of Theorem
\ref{mainTheorem2}.

Let us also note that (\ref{alphaConnEq}) is a consequence of Theorem
\ref{mainTheorem2}: On the one hand Theorem \ref{mainTheorem2} implies that 
\[
\alpha(2n+1;n+1,n,n,\ldots,1,1) = (-1)^{n}A_{n+1,1} = (-1)^{n} \alpha(n;1,2,\ldots,n).
\]
On the other hand -- since the penultimate row of a DMT with bottom row
$(n+1,n,n,\ldots,1,1)$ is $(n,n,n-1,n-1,\ldots,1,1)$ -- Lemma
\ref{sumOpDMT} yields
\[
\alpha(2n+1;n+1,n,n,\ldots,1,1) = (-1)^{n} \alpha(2n;n,n,n-1,n-1,\ldots,1,1).
\]

Following the sketched methodology, the proof of Theorem \ref{wniTheorem} will
consist of the following parts: In Lemma \ref{wniSymmetry} we show that the
numbers $W_{n,i}$ satisfy a certain symmetry relation by giving them a
combinatorial meaning. In Lemma \ref{wniLES} we derive that $(W_{n,i})_{i=1,\ldots,2n-1}$ is an
eigenvector of \linebreak[4] $\left(\binom{n-i}{2n-i-j}(-1)^{n+i}
\right)_{1\leq i,j \leq 2n-1}$. In Lemma \ref{eigenspaceDimLemma} it is shown that the corresponding eigenspace is one-dimensional. In Lemma \ref{conjNumbersEVLemma} we show that
$\left(\sum_{l=1}^{i} \binom{i-1}{l-1} (-1)^{n+i+l-1}
A_{n,l}\right)_{i=1,\ldots,2n-1}$ is in the same eigenspace. In Lemma
\ref{wniRecursionLemma} we show a recursion for $W_{n,1}$, which lets us then
inductively derive the constant factor.

The symmetry $A_{n,i} = A_{n+1-i}$ satisfied by the refined ASM numbers is a
direct consequence of the involution reflecting an ASM along the
horizontal symmetry axis. The idea for proving the symmetry $W_{n,i}=(-1)^{n-1}
W_{n,2n-i}$ is to give the numbers a combinatorial interpretation (as signed
enumeration) and to construct a one-to-one correspondence between the objects
enumerated by $W_{n,i}$ and those enumerated by $W_{n,2n-i}$. In this case a
sign change of $(-1)^{n-1}$ is involved.
\begin{lemma}
\label{wniSymmetry}
The numbers $W_{n,i}$ satisfy the symmetry relation
\[
W_{n,i}=(-1)^{n-1} W_{n,2n-i}
\]
for $i=1,\ldots,2n-1$, $n \geq 1$.
\end{lemma}
\begin{proof}
The summation operator's recursion (\ref{sumOpRec3}) implies that
\begin{multline*}
\Delta_{k_1} \sum_{(l_1,\ldots,l_{n-1})}^{(k_1,\ldots,k_n)}
A(l_1,\ldots,l_{n-1}) \\
= \Delta_{k_1} \left(-
\sum_{(l_2,\ldots,l_{n-1})}^{(k_2,\ldots,k_n)}
\sum_{l_1=k_2}^{k_1-1}
A(l_1,\ldots,l_{n-1})+\sum_{(l_2,\ldots,l_{n-1})}^{(k_2+1,k_3,\ldots,k_n)}
A(k_2,l_2,\ldots,l_{n-1}) \right) \\
= - \sum_{(l_2,\ldots,l_{n-1})}^{(k_2,\ldots,k_n)}
A(k_1,l_2,\ldots,l_{n-1}).
\end{multline*}
Using recursion (\ref{alphaRec}) together with Lemma \ref{sumOpDMT} shows for
weakly decreasing sequences $k_1 \geq k_2 \geq \cdots \geq k_n$ that
\[
\Delta_{k_1} \alpha(n; k_1,\ldots,k_n) = -\sum_{(l_2,\ldots,l_{n-1})\in
\mathcal{P}(k_2,\ldots,k_n)}
(-1)^{\SC(k_2,\ldots,k_n;l_2,\ldots,l_{n-1})}\alpha(n-1;
k_1,l_2,\ldots,l_{n-1}),
\]
where $\SC(k_2,\ldots,k_n; l_2,\ldots,l_{n-1})$ is the
total number of pairs and newcomers in $(l_2,\ldots,l_{n-1})$.

A \emph{DMT-trapezoid} with bottom row $(k_1,\ldots,k_n)$ and top row
$(l_1,\ldots,l_{n-i+1})$ is an array of integers $(a_{p,q})_{\substack{1 \leq p
\leq i,\\ 1 \leq q \leq n-i+p}}$ having the same properties as DMTs. Let
$\mathcal{P}_i(\mathbf{k};\mathbf{l})$ denote the set of DMT-trapezoids with
$i$ rows, bottom row $\mathbf{k}$ and top row $\mathbf{l}$. As the
$\Delta$-operator is linear and the sequence $(k_1,l_2,\ldots,l_{n-1})$ is again
weakly decreasing, we can apply induction to see that
\begin{multline}
\Delta_{k_1}^{i-1} \alpha(n;k_1,\ldots,k_n) \\
= (-1)^{i-1} \sum_{\substack{(l_1,\ldots,l_{n-i},\strap): \\ \strap \in
\mathcal{P}_i(k_2,\ldots,k_n;l_1,\ldots,l_{n-i})}} (-1)^{\SC(\strap)}
\alpha(n-i+1;k_1,l_1,l_2,\ldots,l_{n-i})
\end{multline}
holds for $1 \leq i \leq n$, where $\SC(\strap)$ denotes the total number of
sign changes in the trapezoid, i.e. the total number of pairs and newcomers without
the bottom row. Together with Corollary \ref{mainProofStepOne} we have
\begin{multline*}
W_{n,i} = (-1)^{i-1} \sum_{\substack{(l_1,\ldots,l_{2n-i-1},\strap): \\ \strap \in
\mathcal{P}_i(n-1,n-1,\ldots,1,1;l_1,\ldots,l_{2n-i-1})}} (-1)^{\SC(\strap)}
\alpha(2n-i;n,l_1,l_2,\ldots,l_{2n-i-1}) \\
= \sum_{\substack{(l_1,\ldots,l_{2n-i-1},\strap): \\ \strap \in
\mathcal{P}_i(n-1,n-1,\ldots,1,1;l_1,\ldots,l_{2n-i-1})}}
\sum_{\bigtriangleup \in \mathcal{W}_{2n-i}(n,l_1,l_2,\ldots,l_{2n-i-1})}
(-1)^{i-1+\SC(\strap)+\SC(\bigtriangleup)},
\end{multline*}
i.e. $W_{n,i}$ is a signed enumeration of DMT-trapezoids with $i$ rows,
bottom row $(n-1,n-1,\ldots,1,1)$, where in the top row an entry $n$ is
added at the left end, and completed to the top as DMT. Note that the $(2n-i)$-th row is
special: Firstly, it contains the additional entry $n$ (and this is the only
one, since the entry to the right is strictly smaller). Secondly, it contains the
same number of entries smaller than $n$ as the row above. Apart from this
difference concerning the shape of the array, the restrictions are the same as
for DMTs.

From the one-to-one correspondence with $2$-ASMs (see Theorem
\ref{dmtBijection}) it follows that these objects are in one-to-one correspondence with the following
$(2n-1)\times(n-1)$-matrices: all columns are generated by the machine in Figure \ref{dmtCol} (with length
$2n-1$), and all rows -- with the exception of the $(2n-i)$-th row -- are
generated by the machine in Figure \ref{asmCol} (with length $n-1$). The $(2n-i)$-th row is generated by
the machine in Figure \ref{asmCol} (with length $n-1$) with the modification
that the transition to the end state is at the state $\Sigma = 0$ (this
corresponds to the fact that there is no additional entry $\leq n-1$ in this
row and that this entry is missing at the left end of the row). An example can
be seen in Figure \ref{wniExample}.
\begin{figure}[ht]
\begin{center}
$
\begin{array}{ccccccccccccc}
&&&&& 2 \\
&&&& 2 && 2 \\
&&& 3 && 2 && 2 \\
&& 3 && 3 && 2 && 1 \\
& 4 && 3 && 3 && 1 && 1 \\
&& 3 && 3 && 2 && 1 && 1 \\
& 3 && 3 && 2 && 2 && 1 && 1 \\
\end{array}
$
$
\Leftrightarrow
\left(
\begin{array}{ccc}
0 & 1 & 0 \\
0 & 1 & 0 \\
0 & 0 & 1 \\
1 & -1 & 1 \\
1 & -1 & 0 \\
0 & 1 & 0 \\
0 & 1 & 0 \\
\end{array}
\right)
$
\end{center}
\caption{One of the objects enumerated by $W_{4,3}$ and its corresponding
matrix.}
\label{wniExample}
\end{figure}

The advantage of this vantage point is that the one-to-one correspondence
between the objects enumerated by $W_{n,i}$ and those enumerated by $W_{n,2n-i}$ is now
obvious, namely reflecting the corresponding matrices along the horizontal
symmetry axis.

We now need to know how $\SC(\strap)+\SC(\bigtriangleup)$ changes under this
reflection: From
the vantage point of matrices the total number of pairs
is equal to the number of times we take an edge leading into the state
$\Sigma=2$ of the column-generating machine in Figure \ref{dmtCol}. On the other hand, the
number of newcomers is equal to the number of positions in the matrix, where the
partial column sum is $1$ and the entry below is a $-1$. Hence it is given by
the total number of times we take the $(-1-1)$-edge or the $(1 -1)$-edge. Since
the $\SC$-function counts the number of newcomers and pairs without the bottom
row, it follows that $\SC(\strap)+\SC(\bigtriangleup)+n-1$ is equal to the total
number of edges taken in all columns except for the $0$-loop at the state $\Sigma=0$.

Note that reflecting the matrix along the horizontal symmetry axis means that
the columns of the reflected object are generated in reverse order, i.e. the
number of times the edges are taken is interchanged in the following way:
$(1-1) \leftrightarrow (-11)$, $0$-loop at state $\Sigma = 0$ $\leftrightarrow$
$0$-loop at state $\Sigma = 2$. Hence the parity of the difference between the
number of times we take the two $0$-loops gives us the change of
$(-1)^{\SC(\strap)+\SC(\bigtriangleup)}$ under the reflection. Since the parity
of the total number of times we take the $0$-loops is equal to the parity of the
total number of entries in the matrix (which is equal to the parity of $n-1$), we
have
\begin{multline*}
W_{n,i} = (-1)^{n-1} \sum_{\substack{(l_1,\ldots,l_{i-1},\strap): \\ \strap \in
\mathcal{P}_{2n-i}(n-1,n-1,\ldots,1,1;l_1,\ldots,l_{i-1})}}
\sum_{\bigtriangleup \in \mathcal{W}_{i}(n,l_1,l_2,\ldots,l_{i-1})}
(-1)^{2n-i-1+\SC(\strap)+\SC(\bigtriangleup)} \\
= (-1)^{n-1} W_{n,2n-i}. 
\end{multline*}
\end{proof}
\begin{lemma}
\label{wniLES}
Fix $n \geq 1$. The numbers $W_{n,i}$ satisfy the system of linear equations
\[
W_{n,i} = \sum_{k=1}^{2n-1} \binom{n-i}{2n-i-k} (-1)^{n+i} W_{n,k},
\]
for $i=1,\ldots,2n-1$.
\end{lemma}
\begin{proof}
We will make use of the following properties of the $\alpha$-polynomial
\begin{enumerate}
  \item $\alpha(n;k_1,k_2,\ldots,k_n) = (-1)^{n-1} \alpha(n;
  k_2,k_3,\ldots,k_n,k_1-n)$
  \item $\alpha(n;k_1,k_2,\ldots,k_n) = \alpha(n;-k_n,-k_{n-1},\ldots,-k_1)$
  \item $\alpha(n;k_1,k_2,\ldots,k_n) = \alpha(n;k_1+c,k_2+c,\ldots,k_n+c),
  \quad c \in \mathbb{Z}$.
\end{enumerate}
The first property has been shown in \cite{FischerNewRefProof}, Lemma $5$. The second and
third property is clear for strictly increasing sequences $k_1 < k_2 <
\cdots < k_n$. The polynomiality implies the equality for arbitrary
$(k_1,\ldots,k_n) \in \mathbb{Z}^n$. It follows that
\begin{multline*}
W_{n,i} = \Delta_{k_1}^{i-1} \alpha(2n-1;
n-1,n-1,n-2,n-2,\ldots,1,1,k_1-2n+1)|_{k_1=n}\\
= E_{k_1}^{i-n} \delta_{k_1}^{i-1} \alpha(2n-1;
n-1,n-1,n-2,n-2,\ldots,1,1,k_1-n)|_{k_1=n},
\end{multline*}
where $\delta_x f(x) := f(x)-f(x-1) = (\id - E_x^{-1}) f(x)$ = $E_x^{-1}
\Delta_x f(x)$. Using $E_x = (\id - \delta_x)^{-1}$, the Binomial Theorem and
the fact that applying the $\delta$-operator to a polynomial decreases its degree, we
have
\begin{multline*}
W_{n,i} = \sum_{j=0}^{2n-i-1} \binom{n-i}{j} (-1)^j \delta_{k_1}^{i+j-1}
\alpha(2n-1;n-1,n-1,n-2,n-2,\ldots,1,1,k_1)|_{k_1=0} \\
= \sum_{j=0}^{2n-i-1} \binom{n-i}{j} (-1)^j \delta_{k_1}^{i+j-1}
\alpha(2n-1;-k_1,-1,-1,-2,-2,\ldots,-(n-1),-(n-1))|_{k_1=0}.  
\end{multline*}
Note by induction that the operators $\delta$ and $\Delta$ satisfy the equation
\[
\delta_x^i f(-x) = (-1)^i \Delta_y^i f(y) |_{y=-x}
\]
for any function $f$ and $i\geq 0$. Hence,
\begin{multline*}
W_{n,i} = \sum_{j=0}^{2n-i-1} \binom{n-i}{j} (-1)^{i-1} \Delta_{k_1}^{i+j-1}
\alpha(2n-1;k_1,-1,-1,-2,-2,\ldots,-(n-1),-(n-1))|_{k_1=0} \\
= \sum_{j=0}^{2n-i-1} \binom{n-i}{j} (-1)^{i-1} \Delta_{k_1}^{i+j-1}
\alpha(2n-1;k_1,n-1,n-1,n-2,n-2,\ldots,1,1)|_{k_1=n} \\
= \sum_{j=0}^{2n-i-1} \binom{n-i}{j} (-1)^{i-1} W_{n,i+j} = \sum_{k=i}^{2n-1} \binom{n-i}{k-i} (-1)^{i-1} W_{n,k}.
\end{multline*}
Using the symmetry shown in Lemma \ref{wniSymmetry} yields the claimed equation:
\[
W_{n,i} = \sum_{k=i}^{2n-1} \binom{n-i}{k-i} (-1)^{n+i} W_{n,2n-k} = \sum_{k=1}^{2n-1} \binom{n-i}{2n-i-k} (-1)^{n+i} W_{n,k}.
\]
\end{proof}
\begin{lemma}
\label{eigenspaceDimLemma}
Let $\delta_{i,j}$ denote the Kronecker delta. Then
\[
\rk\left(\binom{n-i}{2n-i-j}(-1)^{n+i} - \delta_{i,j} \right)_{1\leq i,j \leq
2n-1} = 2n-2,
\]
holds for $n\geq 1$, i.e. the eigenspace of the eigenvalue $1$ of
$\left(\binom{n-i}{2n-i-j}(-1)^{n+i} \right)_{1\leq i,j \leq 2n-1}$ is one-dimensional.
\end{lemma}
\begin{proof}
Note that the $n$-th row of $S:=\left(\binom{n-i}{2n-i-j}(-1)^{n+i} -
\delta_{i,j} \right)_{1 \leq i,j \leq 2n-1}$ contains only $0$-entries. In the following, it
will be shown that removing this row and the last column results in a $(2n-2)\times(2n-2)$-matrix $S'$ with non-zero determinant (in fact, it is equal to $(-1)^{n-1} A_{n-1}$, where $A_{n-1}$ denotes the
number of ASMs of size $n-1$). The block structure of $S'$ is displayed in
Figure \ref{matrixBlockStructure}.
\begin{figure}[ht]
$
\begin{array}{c|c|c|}
\multicolumn{3}{c}{\hfill j=1,\ldots,2n-2 \qquad\qquad\qquad\qquad} \\
\cline{2-3}
i=1,\ldots,n-1 & \begin{array}{ccccc}-1 & 0 && \cdots & 0 \\
 0 & -1 &\ddots&& \vdots \\ \vdots &\ddots&\ddots&& \\ &&& -1 & 0 \\ 0 && \cdots
 &0& -1 \end{array} & \binom{n-i}{2n-i-j}(-1)^{n+i} \\
\cline{2-3}
i=n+1,\ldots,2n-1 & \binom{n-i}{2n-i-j}(-1)^{n+i} & \begin{array}{ccccc}0 & -1
&& \cdots & 0 \\ & 0 &\ddots&& \vdots \\ \vdots &&\ddots&& \\ &&& 0 & -1 \\ 0 && \cdots && 0
 \end{array} \\
 \cline{2-3}
\end{array}
$
\caption{The matrix $S'$ decomposed into four $(n-1)\times(n-1)$-blocks.
\label{matrixBlockStructure}}
\end{figure}
The determinant of
a block matrix $\block{$A$}{$B$}{$C$}{$D$}$ with an invertible matrix $A$
and a square matrix $D$ is equal to $\det(A)\det(D-CA^{-1}B)$: This follows from the
decomposition 
\[
\block{$A$}{$B$}{$C$}{$D$} = \block{$A$}{$0$}{$C$}{$I$}
\block{$I$}{$A^{-1}B$}{$0$}{$D-CA^{-1}B$}
\]
together with the fact that
\[
\det \block{$A$}{$0$}{$C$}{$D$} = \det \block{$A$}{$B$}{$0$}{$D$} =
\det (A) \det (D).
\]
The block matrices in our case are
\begin{align*}
A &= -I_{n-1}, \\
B&=\left(
\binom{n-i}{n-i-j+1}(-1)^{n+i} \right)_{1 \leq i,j \leq n-1}, \\
C&=\left(\binom{-i}{n-i-j}(-1)^i\right)_{1\leq i,j \leq n-1}, \\
D &= \left(-\delta_{i+1,j} \right)_{1 \leq i,j \leq n-1}.
\end{align*}
Applying Chu-Vandermonde summation shows that the matrix $C$ is invertible with
inverse $C^{-1}=\left( \binom{i-n}{i+j-n} (-1)^{n-i} \right)_{1 \leq i,j \leq
n-1}$.
It follows that
\[
\det(S') = \det(A) \det(C) \det(C^{-1}D - A^{-1}B). 
\]
The determinant of $A$ is obviously equal to $(-1)^{n-1}$. An easy calculation
shows that the determinant of $C$ is as well equal to $(-1)^{n-1}$. As $A^{-1} =
-I_{n-1}$ we further have
\begin{multline*}
\det(S') = \det(C^{-1}D + B) \\
= \det_{1 \leq i,j \leq n-1} \left( \binom{i-n}{i+j-1-n} (-1)^{n+i+1} +
\binom{n-i}{n-i-j+1}(-1)^{n+i} \right).
\end{multline*}
Using the identity $\binom{n}{k} = \binom{n}{n-k}$ -- which is true for arbitrary
$k \in \mathbb{Z}$, whenever $n \in \mathbb{Z}^+$ -- the second
binomial coefficient is equal to $\binom{n-i}{j-1}$. Multiplying the $i$-th row with $(-1)^{n+i}$ yields
\[
\det(S') = (-1)^{\binom{n}{2}} \det_{1 \leq i,j \leq n-1} \left(
\binom{n-i}{j-1}-\binom{i-n}{i+j-1-n} \right).
\]
Switching row $i$ with row $n-i$ for $i=1,\ldots,\floor{\frac{n-1}{2}}$ and
noting that $(-1)^{\binom{n}{2} + \floor{\frac{n-1}{2}}} = (-1)^{n-1}$ gives
\[
\det(S') = (-1)^{n-1} \det_{1 \leq i,j \leq n-1} \left(
\binom{i}{j-1}-\binom{-i}{j-i-1} \right).
\]
Multiplying from the right with the upper-triangular matrix
$\left(\binom{j-2}{j-i}\right)_{1\leq i,j \leq n-1}$ having determinant $1$ and
using Chu-Vandermonde convolution yields
\begin{multline*}
\det(S') = (-1)^{n-1} \det_{1 \leq
i,j \leq n-1} \left(\sum_{k=1}^{n-1}
\left(\binom{i}{k-1}-\binom{-i}{k-i-1}\right)\binom{j-2}{j-k} \right) \\
=(-1)^{n-1} \det_{1 \leq i,j \leq n-1}
\left(\binom{i+j-2}{j-1}-\binom{j-i-2}{j-i-1} \right)\\
=(-1)^{n-1} \det_{0 \leq i,j \leq n-2}
\left(\binom{i+j}{j}-\delta_{i,j+1} \right). \\
\end{multline*}
In \cite{BehrendWeightedEnum} it has been shown that the determinant of
$\left(\binom{i+j}{j}-\delta_{i,j+1} \right)_{0 \leq i,j \leq n-1}$ is equal to the
number of descending plane partitions with each part smaller than or equal to
$n$, which is known to be equal to the number of ASMs of size $n$. It follows that
\[
\det(S')=(-1)^{n-1} A_{n-1} \neq 0.
\]
\end{proof}
\begin{lemma}
\label{conjNumbersEVLemma}
Let $X_{n,i}:=\sum_{j=1}^{i}
\binom{i-1}{j-1}(-1)^{n+i+j-1}A_{n,j}$. Then
\[
\sum_{k=1}^{2n-1} \binom{n-i}{2n-i-k}(-1)^{n+i} X_{n,k} = X_{n,i}
\]
holds for all $i=1,\ldots,2n-1$, $n\geq 1$.
\end{lemma}
\begin{proof}
The refined ASM-numbers $A_{n,j}$ are equal to
\[
A_{n,j}=\binom{n+j-2}{n-1}\frac{(2n-j-1)!}{(n-j)!} \prod_{k=0}^{n-2}
\frac{(3k+1)!}{(n+k)!} = \binom{n+j-2}{n-1}\binom{2n-j-1}{n-1} c_n
\]
with a constant $c_n$ independent of $j$. Hence, it suffices to show that
\begin{multline*}
\sum_{k=1}^{2n-1}\binom{n-i}{2n-i-k} \sum_{j=1}^k
\binom{k-1}{j-1}\binom{n+j-2}{n-1}\binom{2n-j-1}{n-1} (-1)^{i+j+k-1} \\=
\sum_{j=1}^{2n-1}
\binom{i-1}{j-1}\binom{n+j-2}{n-1}\binom{2n-j-1}{n-1}(-1)^{n+i+j-1}.
\end{multline*}
Using $\binom{n}{k} = \binom{n}{n-k}$, $\binom{a}{b} =
\binom{b-a-1}{b} (-1)^b$ and Chu-Vandermonde convolution shows that the
left-hand side is equal to
\begin{multline*}
\sum_{j=1}^{2n-1} \binom{n+j-2}{n-1}\binom{2n-j-1}{n-1} (-1)^{i+j-1}
\sum_{k=j}^{2n-1} \binom{n-i}{2n-i-k} \binom{k-1}{j-1} (-1)^k \\
= \sum_{j=1}^{2n-1} \binom{n+j-2}{n-1}\binom{2n-j-1}{n-1} (-1)^{i-1}
\sum_{k=j}^{2n-1} \binom{n-i}{2n-i-k} \binom{-j}{k-j} \\
= \sum_{j=1}^{2n-1} \binom{n+j-2}{n-1}\binom{2n-j-1}{n-1} \binom{n-i-j}{2n-i-j}
(-1)^{i-1}.
\end{multline*}
Hence, we have to show that
\begin{multline*}
\left(\sum_{j=1}^{2n-1} \binom{n+j-2}{n-1}\binom{2n-j-1}{n-1}
\binom{n-i-j}{2n-i-j} \right)_{i=1,\ldots,2n-1} \\= \left( \sum_{j=1}^{2n-1}
\binom{i-1}{j-1}\binom{n+j-2}{n-1}\binom{2n-j-1}{n-1}(-1)^{n+j}
\right)_{i=1,\ldots,2n-1}.
\end{multline*}
Multiplying with the invertible matrix $T:=\left( \binom{i-1}{j-1} (-1)^j
\right)_{1 \leq i,j \leq 2n-1}$ from the left transforms the right-hand side
into
\begin{multline*}
\left( \sum_{k=1}^{2n-1} \sum_{j=1}^{2n-1}
\binom{k-1}{j-1}\binom{n+j-2}{n-1}\binom{2n-j-1}{n-1}(-1)^{n+j}
\binom{i-1}{k-1}(-1)^k \right)_{i=1,\ldots,2n-1} \\
= \left( \sum_{j=1}^{2n-1}
\binom{n+j-2}{n-1}\binom{2n-j-1}{n-1} (-1)^{n} \sum_{k=1}^{2n-1}
\binom{-j}{k-j} \binom{i-1}{i-k} \right)_{i=1,\ldots,2n-1} \\
= \left( \sum_{j=1}^{2n-1}
\binom{n+j-2}{n-1}\binom{2n-j-1}{n-1} (-1)^{n} \binom{i-j-1}{i-j}
\right)_{i=1,\ldots,2n-1} \\
= \left( \binom{n+i-2}{n-1}\binom{2n-i-1}{n-1} (-1)^{n}
\right)_{i=1,\ldots,2n-1}, \\
\end{multline*}
where the last equality is due to $\binom{k-1}{k} = \delta_{k,0}$,
$k\in\mathbb{Z}$. After multiplying with matrix $T$ from the left, the left-hand
side is equal to
\begin{multline*}
\left(\sum_{k=1}^{2n-1} \sum_{j=1}^{2n-1} \binom{n+j-2}{n-1}\binom{2n-j-1}{n-1}
\binom{n-k-j}{2n-k-j} \binom{i-1}{k-1} (-1)^k \right)_{i=1,\ldots,2n-1} \\
= \left(\sum_{j=1}^{2n-1} \binom{n+j-2}{n-1}\binom{2n-j-1}{n-1} (-1)^j
\sum_{k=1}^{2n-1} \binom{n-1}{2n-k-j} \binom{i-1}{k-1}
\right)_{i=1,\ldots,2n-1} \\
= \left(\sum_{j=1}^{2n-1} \binom{n+j-2}{n-1}\binom{2n-j-1}{n-1}
\binom{n+i-2}{2n-j-1} (-1)^j \right)_{i=1,\ldots,2n-1}. \\
\end{multline*}
Note that
$
\binom{n+j-2}{n-1}\binom{2n-j-1}{n-1} \binom{n+i-2}{2n-j-1} =
\binom{n+j-2}{j-1}\binom{i-1}{n-j}\binom{n+i-2}{n-1}
$
holds for $i,j = 1,\ldots,2n-1$. The left-hand side is hence equal to
\begin{multline*}
\left( - \binom{n+i-2}{n-1} \sum_{j=1}^{2n-1}
\binom{-n}{j-1}\binom{i-1}{n-j} \right)_{i=1,\ldots,2n-1} 
= \left( - \binom{n+i-2}{n-1} \binom{i-n-1}{n-1} \right)_{i=1,\ldots,2n-1} 
\\
= \left(\binom{n+i-2}{n-1} \binom{2n-i-1}{n-1} (-1)^n
\right)_{i=1,\ldots,2n-1}.
\end{multline*}
\end{proof}
\begin{lemma}
\label{wniRecursionLemma}
The numbers $W_{n,i}$ satisfy the equation
\[
W_{n,1} = -\sum_{i=1}^{n-1} \binom{n-1}{i} W_{n-1,i}
\]
for all $n \geq 2$.
\end{lemma}
\begin{proof}
Using Lemma
\ref{sumOpDMT}
and $\alpha(n;k_1,\ldots,k_n)=(-1)^{n-1}\alpha(n;k_n+n,k_1,\ldots,k_{n-1})$
(compare to the first property of $\alpha$ in the proof of Lemma \ref{wniLES}),
we have
\begin{multline*}
W_{n,1} = \alpha(2n-1; n, n-1,n-1,\ldots,1,1) = (-1)^{n-1} \alpha(2n-2; n-1,n-1,\ldots,1,1) \\
= (-1)^n \alpha(2n-2; 2n-1,n-1,n-1,\ldots,2,2,1).
\end{multline*}
A DMT with bottom row $(2n-1,n-1,n-1,\ldots,2,2,1)$ has the penultimate row
$(l,n-1,n-1,\ldots,2,2)$, where $l$ may take any value in
$\{n,n+1,\ldots,2n-2\}$. Filling in the penultimate row causes one sign-change
for each pair $\{(2,2),(3,3),\ldots,(n-1,n-1)\}$ and one sign-change for the
entry $l$ being a newcomer.
Together with the binomial identity $\sum\limits_{k=n}^m \binom{k}{n} =
\binom{m+1}{n+1}$, it follows that
\begin{multline*}
W_{n,1} = - \sum_{l=n}^{2n-2} \alpha(2n-3;l,n-1,n-1,\ldots,2,2) \\
= -\sum_{l=n-1}^{2n-3} \alpha(2n-3;l,n-2,n-2,\ldots,1,1) \\
= -\sum_{l=n-1}^{2n-3} (\Delta_{k_1}+\id)^{l-n+1}
\alpha(2n-3;k_1,n-2,n-2,\ldots,1,1)|_{k_1=n-1} \\
= -\sum_{l=n-1}^{2n-3} \sum_{i = 0}^{l-n+1} \binom{l-n+1}{i} \Delta_{k_1}^i
\alpha(2n-3;k_1,n-2,n-2,\ldots,1,1)|_{k_1=n-1} \\
= -\sum_{i = 0}^{n-2} W_{n-1,i+1} \sum_{l=n+i-1}^{2n-3} \binom{l-n+1}{i} =
-\sum_{i = 1}^{n-1} \binom{n-1}{i} W_{n-1,i}.
\end{multline*}
\end{proof}
We are now in the position to conclude the proof of Theorem \ref{wniTheorem}:
For this, let us prove by induction on $n$, that $X_{n,i} = W_{n,i}$ for all $i =
1,\ldots,2n-1$, $n \geq 1$. As both $(X_{n,i})_{i=1,\ldots,2n-1}$ and $(W_{n,i})_{i=1,\ldots,2n-1}$ are in the same
one-dimensional eigenspace for fixed $n \geq 1$, it follows that $(X_{n,i})_{i=1,\ldots,2n-1} = C_n
(W_{n,i})_{i=1,\ldots,2n-1}$. Hence, it suffices to show that $X_{n,1} =
W_{n,1}$. 

The case $n=1$ is trivial. For $n \geq 2$, apply Lemma
\ref{wniRecursionLemma} and the induction hypothesis in order to get
\begin{multline*}
W_{n,1} = -\sum_{i=1}^{n-1} \binom{n-1}{i} W_{n-1,i} = -\sum_{i=1}^{n-1} \binom{n-1}{i} X_{n-1,i} \\
= -\sum_{i=1}^{n-1} \binom{n-1}{i} \sum_{j=1}^{i}
\binom{i-1}{j-1}(-1)^{n+i+j}A_{n-1,j} \\
= (-1)^{n-1}
\sum_{j=1}^{n-1} A_{n-1,j} \sum_{i=j}^{n-1} \binom{n-1}{i}
\binom{i-1}{j-1}(-1)^{i+j}.
\end{multline*}
Applying Vandermonde convolution to the inner sum shows that it is equal to $1$
for all $j=1,\ldots,n-1$. Hence we have
\[
W_{n,1} = (-1)^{n-1} \sum_{j=1}^{n-1} A_{n-1,j} = (-1)^{n-1} A_{n,1} = X_{n,1}.
\]

\section{\sts Towards a bijective proof and further research}
\label{sectionOutlook}

By Theorem \ref{mainTheorem}, we have
\[
\alpha(2n; n,n,n-1,n-1,\ldots,1,1) = (-1)^{\binom{2n}{2}} \sum_{A \in \mathcal{W}_{2n}(n,n,n-1,n-1,\ldots,1,1)} (-1)^{\dd(A)}
\]
Since there are exactly $n$ pairs in the bottom row, and $\binom{2n}{2}$ has the
same parity as $n$, it follows that proving equation (\ref{alphaConnEq}) is
equivalent to showing that
\begin{equation}
\label{conjBij}
\sum_{A \in \mathcal{W}_{2n}(n,n,n-1,n-1,\ldots,1,1)} (-1)^{\overline{dd}(A)} =
\alpha(n;1,2,\ldots,n), \quad n \geq 1,
\end{equation}
where $\overline{dd}(A)$ is the number of pairs $(x,x)$, for which there also exists
a pair $(x,x)$ in the row below. 

A bijective proof of (\ref{conjBij}) could succeed by partitioning
$\mathcal{W}_{2n}(n,n,n-1,n-1,\ldots,1,1)$ into three sets $S_1$,$S_2$ and
$S_3$, such that all elements of $S_1$ as well as $S_2$ have even
$\overline{dd}$-parity, whereas all elements of $S_3$ have odd
$\overline{dd}$-parity. The elements of $S_1$ should be in bijective
correspondence with Monotone Triangles with bottom row
$(1,2,\ldots,n)$, and the elements of $S_2$ and $S_3$ should be in bijective
correspondence.

It is plausible that the subset $S_1 \subseteq
\mathcal{W}_{2n}(n,n,n-1,n-1,\ldots,1,1)$ corresponding to the Monotone
Triangles, is given by those DMT with bottom row $(n,n,n-1,n-1,\ldots,1,1)$,
where the $(2i)$-th row consists of $i$ pairs for all $i = 1,\ldots,n$ (see
Figure \ref{dmt-mt-structure}). Note that this also determines the entries in
odd rows. Identifying the entries connected by an edge with one single entry and
reflecting the triangle along the vertical symmetry axis, yields the one-to-one
correspondence with Monotone Triangles: The condition of DMTs that each row
contains an integer at most twice corresponds to the strict monotony along rows
in Monotone Triangles. The structural restriction of DMTs in $S_1$ ensures that the NSD-condition holds, and the weak monotony along diagonals directly translates from one object to the other. Note that for each $A \in S_1$ we have that
$\overline{dd}(A)$ is twice the number of diagonally adjacent identical entries in the
corresponding Monotone Triangle, hence $(-1)^{\overline{dd}(A)} = 1$.
\begin{figure}
\begin{center}
\includegraphics[width=8cm]{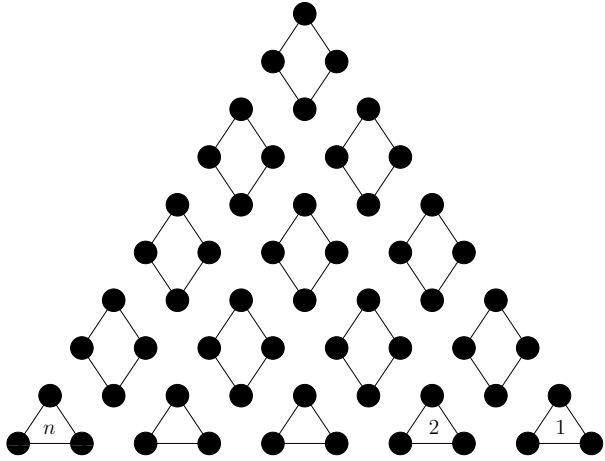}
\caption{Structure-restriction of DMTs such that they correspond to Monotone Triangles: entries connected by an edge must have the same value.} 
\label{dmt-mt-structure}
\end{center}
\end{figure}
Finding a sign-reversing involution on 
$\mathcal{W}_{2n}(n,n,n-1,n-1,\ldots,1,1) \backslash S_1$ remains an open
problem.
 
From the viewpoint of $2$-ASMs the subset $S_1 \subseteq
\mathcal{W}_{2n}(n,n,n-1,n-1,\ldots,1,1)$ corresponds to the set of $2$-ASMs,
where the column generation is restricted to the machine in Figure
\ref{dmtMTCol}. The one-to-one correspondence with Monotone Triangles is even
more obvious from this vantage point. Using this connection might turn out to be useful in finding a
bijection.
\begin{figure}
\begin{center}
\includegraphics[width=8cm]{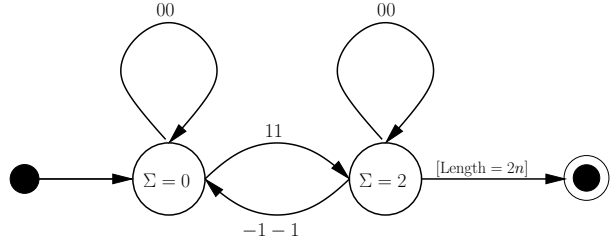}
\caption{Machine generating the columns of those $2$-ASMs corresponding to
$S_1 \subseteq \mathcal{W}_{2n}(n,n,n-1,n-1,\ldots,1,1)$.}
\label{dmtMTCol}
\end{center}
\end{figure}

A second open problem is the following: Considering the shown
equation 
\[
\alpha(n;1,2,\ldots,n) = \alpha(2n;n,n,n-1,n-1,\ldots,1,1),
\] 
it is natural to look for more such nice connections. The generalisation
\[
\alpha(n;k_1,\ldots,k_n) = \alpha(2n;k_n,k_n,\ldots,k_1,k_1)
\] 
to arbitrary strictly increasing sequences $k_1 < k_2 < \cdots < k_n$ does not
hold, for example
\[
\alpha(2;1,4) = 4 \neq 2 = \alpha(4; 4,4,1,1).
\]
Yet, the evaluation $\alpha(n;n,n-1,\ldots,1)$ seems to be of interest. Note
that for an even natural number $n$, the set $\wn(n,n-1,\ldots,1)$ is empty,
and hence Theorem \ref{mainTheorem} implies that \mbox{$\alpha(n;n,n-1,\ldots,1)
= 0$}. For odd numbers $n \geq 3$, the first values of $\alpha(n;
n,n-1,\ldots,1)$ were calculated with Mathematica and are displayed in the
following table:
\begin{center}
\begin{tabular}[h]{|c|c|}
\hline
$\alpha(3;3,2,1)$ & $-1$ \\
\hline
$\alpha(5;5,4,3,2,1)$ & $3$ \\
\hline
$\alpha(7;7,6,5,4,3,2,1)$ & $-26$ \\
\hline
$\alpha(9;9,8,7,6,5,4,3,2,1)$ & $646$ \\
\hline
$\alpha(11;11,10,9,8,7,6,5,4,3,2,1)$ & $-45885$ \\
\hline
$\alpha(13;13,12,11,10,9,8,7,6,5,4,3,2,1)$ & $9304650$ \\
\hline
\end{tabular}
\end{center}
The absolute values of these numbers are known to be the first entries of the
sequence of numbers of Vertically Symmetric Alternating Sign Matrices
(VSASMs). The number of $(2m+1)\times(2m+1)$-VSASM is equal to the number of
Monotone Triangles with prescribed bottom row $(2,4,\ldots,2m)$ leading to
Conjecture \ref{conj}.
\begin{conjecture}
\label{conj}
For $n = 2m+1$, $m\geq 1$, the equation
\begin{equation}
\alpha(n; n,n-1,\ldots,1) = (-1)^{m} \alpha(m;2,4,\ldots,2m)
\end{equation}
seems to hold.
\end{conjecture}
It will be interesting to see, whether similar techniques
as presented here can be applied to give an algebraic proof of the conjecture.

\bibliographystyle{alpha}
\bibliography{bib110825}

\end{document}